\documentclass[leqno,11pt,a4paper]{article}
\usepackage[left=3.2cm, top=2.2cm,bottom=2.5cm]{geometry}
\usepackage{amsmath,amssymb,amsthm,mathrsfs,calc,graphicx}
\usepackage[british]{babel}

\newtheorem {theorem}{Theorem}

\newtheorem {lemma}[theorem]{Lemma}
\newtheorem {corollary}[theorem]{Corollary}

{\theoremstyle{definition}

}
{\theoremstyle{theorem}

}

\def\ba{\begin{array}}
\def\ea{\end{array}}
\def\bea{\begin{eqnarray} \label}
\def\eea{\end{eqnarray}}
\def\be{\begin{equation} \label}
\def\ee{\end{equation}}
\def\bit{\begin{itemize}}
\def\eit{\end{itemize}}
\def\ben{\begin{enumerate}}
\def\een{\end{enumerate}}

\def\NN{\mathbb{N}}

\def\RR{\mathbb{R}}

\def\1{\mathds{1}}

\begin{document}

\title{\bfseries Spatial STIT Tessellations -- Distributional Results for I-Segments}

\author{Christoph Th\"ale\footnotemark[1]\,, Viola Wei\ss\footnotemark[2]\, and Werner Nagel\footnotemark[3]}

\date{}
\renewcommand{\thefootnote}{\fnsymbol{footnote}}
\footnotetext[1]{Universit\"at Osnabr\"uck, Institut f\"ur Mathematik, Albrechtstr. 28a, D-49076 Osnabr\"uck, Germany. Email: christoph.thaele@uni-osnabrueck.de}

\footnotetext[2]{Fachhochschule Jena, Fachbereich Grundlagenwissenschaften, Carl-Zeiss-Promenade 2, D-07745 Jena, Germany. Email: viola.weiss@fh-jena.de}

\footnotetext[3]{Friedrich-Schiller-Universit\"at Jena, Institut f\"ur Stochastik, Enrst-Abbe-Platz 2, D-07743 Jena, Germany. Email: werner.nagel@uni-jena.de}

\maketitle

\begin{abstract}
Three-dimensional random tessellations that are stable under iteration (STIT tessellations) are considered. They arise as a result of subsequent cell division, which implies that their cells are not face-to-face. The edges of the cell-dividing polygons are the so-called I-segments of the tessellation. The main result is an explicit formula for the distribution of the number of vertices in the relative interior of the typical I-segment. On the way of its proof other distributional identities for the typical as well as for the length-weighted typical I-segment are obtained. They provide new insight into the spatio-temporal construction process.\\ \\
\noindent
{\bf Keywords}. {Cell division process; iteration/nesting; marked point process; random tessellation; stability under iteration; stochastic geometry}\\
{\bf MSC}. Primary 60D05; Secondary 60G55, 60E05.
\end{abstract}

\section{Introduction}

In recent years, random tessellation theory has been an active field of research. Whereas in the past mainly mean values and their relations were considered, current research focuses on second-order parameters, limit theorems and distributional results, see \cite{H,HM,HS3} to mention only a few. Besides the classical Poisson hyperplane and Poisson-Voronoi tessellations, random tessellations constructed by subsequent cell division have attracted particular interest in recent times in stochastic geometry and spatial statistics (see \cite{cc,cowan1,rt} and especially \cite{cowan3} and the references cited therein). Among these models, the so-called STIT tessellations (which are {\bf st}able under {\bf it}eration, see below) introduced in \cite{mnwconstr,nwstit} are of particular interest, because of the number of analytically available results \cite{lr,mnwmathnachr,nw3d}, \cite{ST2}--\cite{ST4} and \cite{thaele}--\cite{wc}, see Figure \ref{figstit} for illustrations. The model shows the potential to become a new reference model for crack or fissure structures.

\begin{figure}
\begin{center}
\includegraphics[width=0.49\columnwidth]{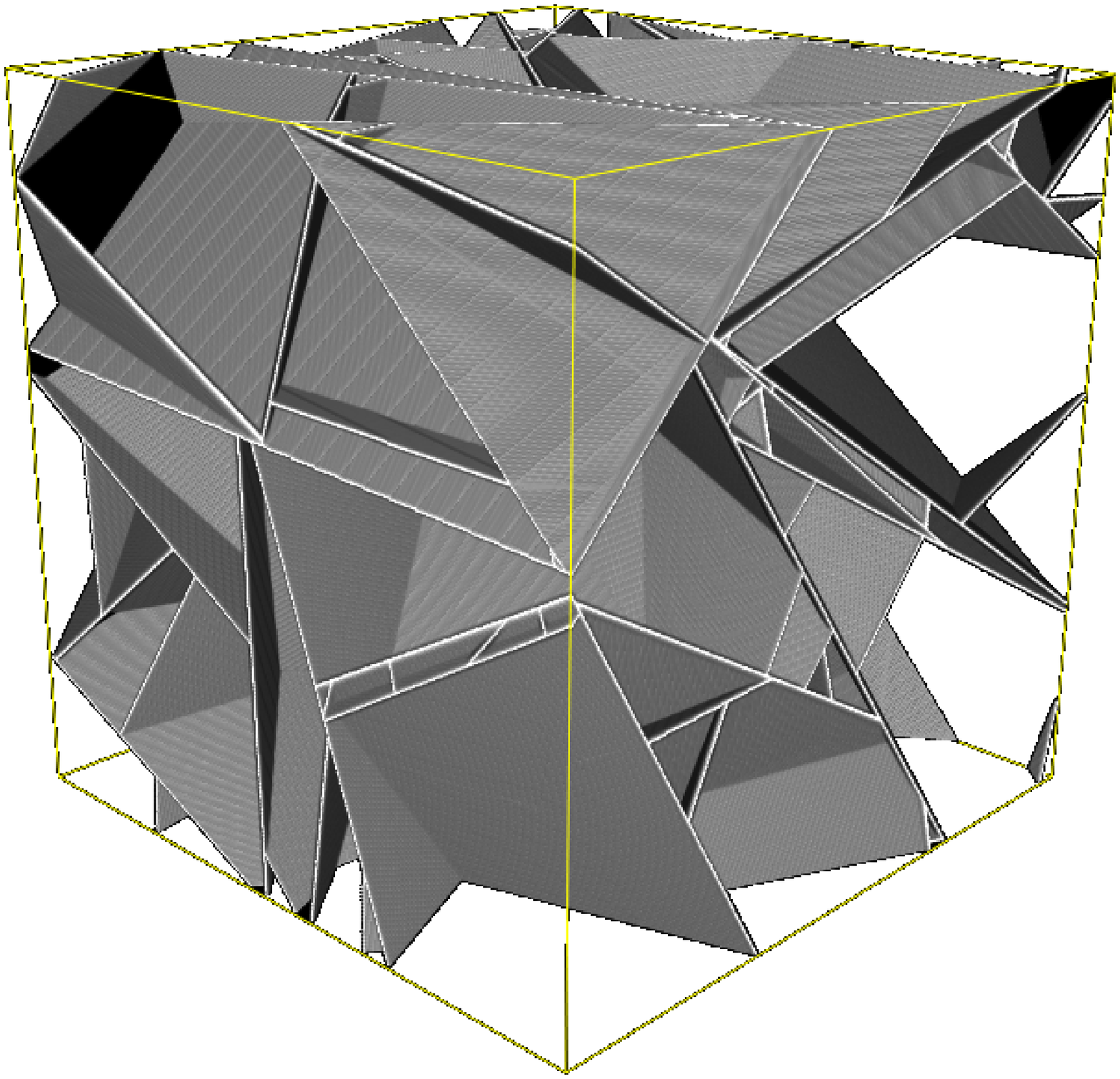}
\includegraphics[width=0.49\columnwidth]{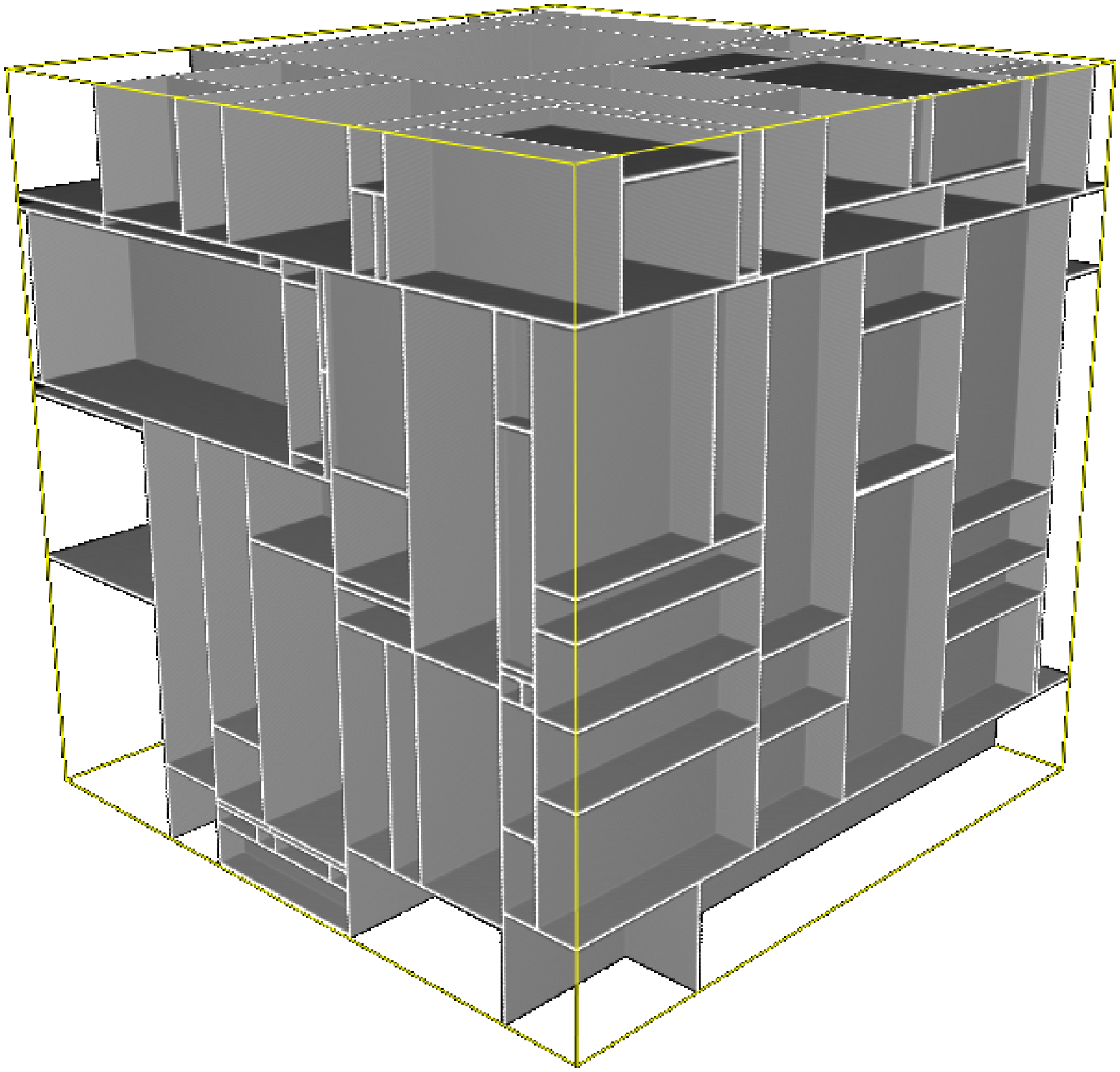}
\caption{Realisations of an isotropic STIT tessellation (left) and of a STIT tessellation whose directional distribution $\cal R$ is concentrated with equal weight on three orthogonal directions (right).}
\label{figstit}
\end{center}
\end{figure}
After a detailed analysis of planar STIT tessellations in \cite{mnwmathnachr,ST2,thaele}, in the present paper  the $3$-dimensional case is considered. We study so-called \textit{I-segments} in homogeneous three-dimensional STIT tessellations, which form the one-dimensional building blocks of a STIT tessellation. They appear in the course of the sequential cell splitting procedure when cells are divided by new planes. In fact, all sides ($1$-faces) of a dividing two-dimensional polygon -- which will be referred to as an \textit{I-polygon} -- are called I-segments. In principle the distribution of the number of vertices in the relative interior of the \textit{typical} I-segment is in the focus of this paper. Regarding the spatio-temporal construction, it becomes evident that this number depends on the birth time, the direction and on the length of the I-segment. But it also depends on the birth time of that I-polygon, in whose interior the I-segment under consideration arises at a later time. This polygon is called the \textit{carrying I-polygon} of the I-segment. At its birth time, this I-polygon is a facet of two adjacent mosaic cells and during the cell division process these cells can undergo further subdivision and in this way new vertices and edges within the carrying I-polygon can appear, see Figure \ref{figpolygon}.

For this reason we consider the edges in the carrying I-polygon (dark grey) and observe that an I-segment can have intersections with already existing edges (light coloured lines in the dark grey polygon)  at the moment of its birth, see Figure \ref{figpolygon}. This is in sharp contrast to the planar case \cite{mnwmathnachr,thaele} and causes that their analysis is considerably more involved. Further vertices can arise after the birth of the typical I-segment as an effect of the ongoing cell division procedure. For deriving an explicit formula for the probability that a fixed number of vertices is located in the relative interior of the typical I-segment, we need to study at first the marked process of I-segments with the following marks: length, direction and birth time of the I-segment and birth time of the carrying I-polygon. We will describe the joint and all marginal distributions of the marks associated to the typical and for the length-weighted typical I-segment of a homogeneous spatial STIT tessellation. Moreover, by noting that such a spatial tessellation has exactly two different types of vertices (see Figure \ref{figknoten}) we can refine our result and obtain the joint distribution of the numbers of vertices of both types in the relative interior of the typical I-segment.

Note that it would also be interesting to calculate the distribution of the number of vertices within the typical $2$-dimensional building block, i.e. the typical I-polygon. However, this quantity is currently not accessible, because to this end the area distribution of a typical Poisson polygon is needed. To determine this  distribution is a long-standing open problem in stochastic geometry. 

\begin{figure}
\begin{center}
\includegraphics[width=0.49\columnwidth]{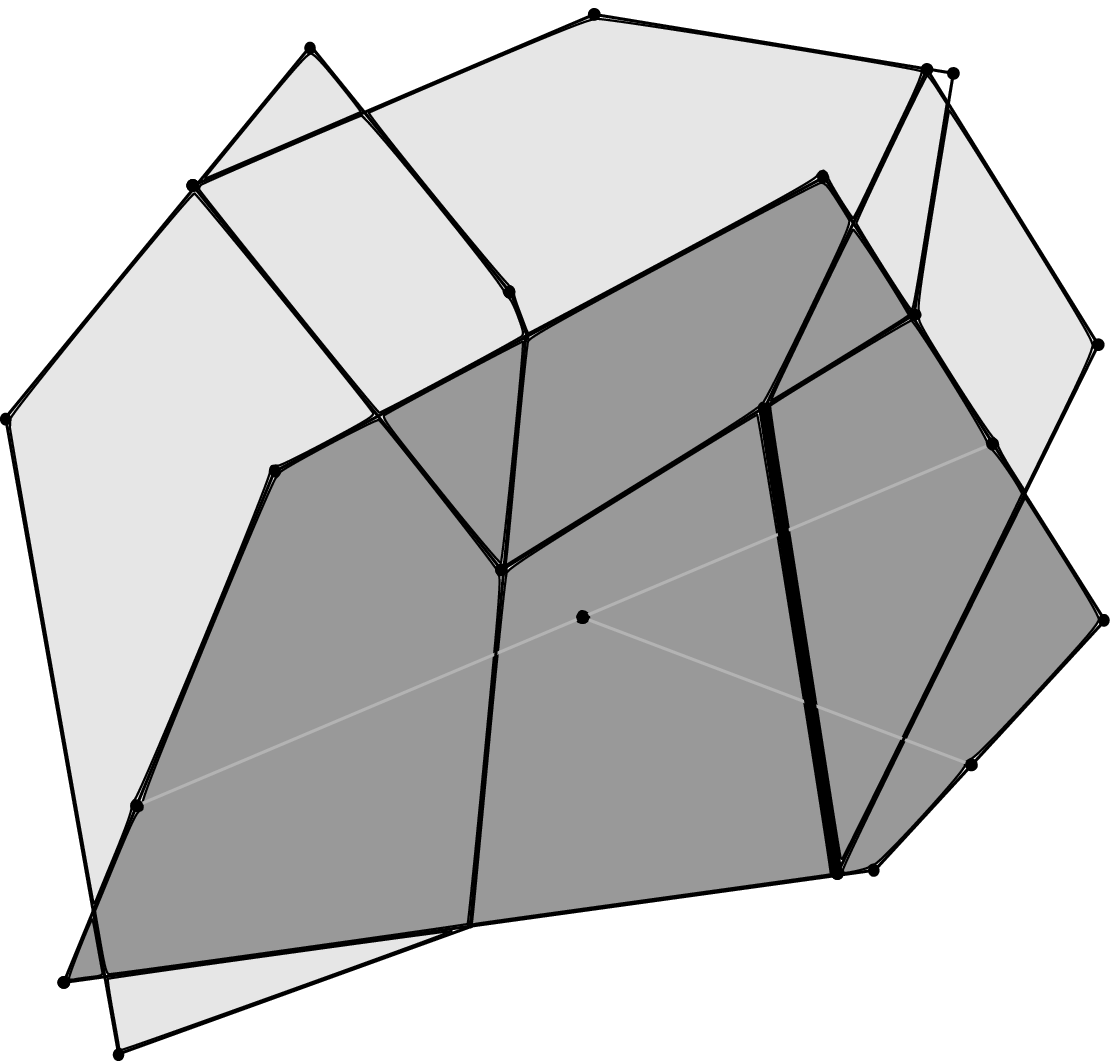}
\includegraphics[width=0.49\columnwidth]{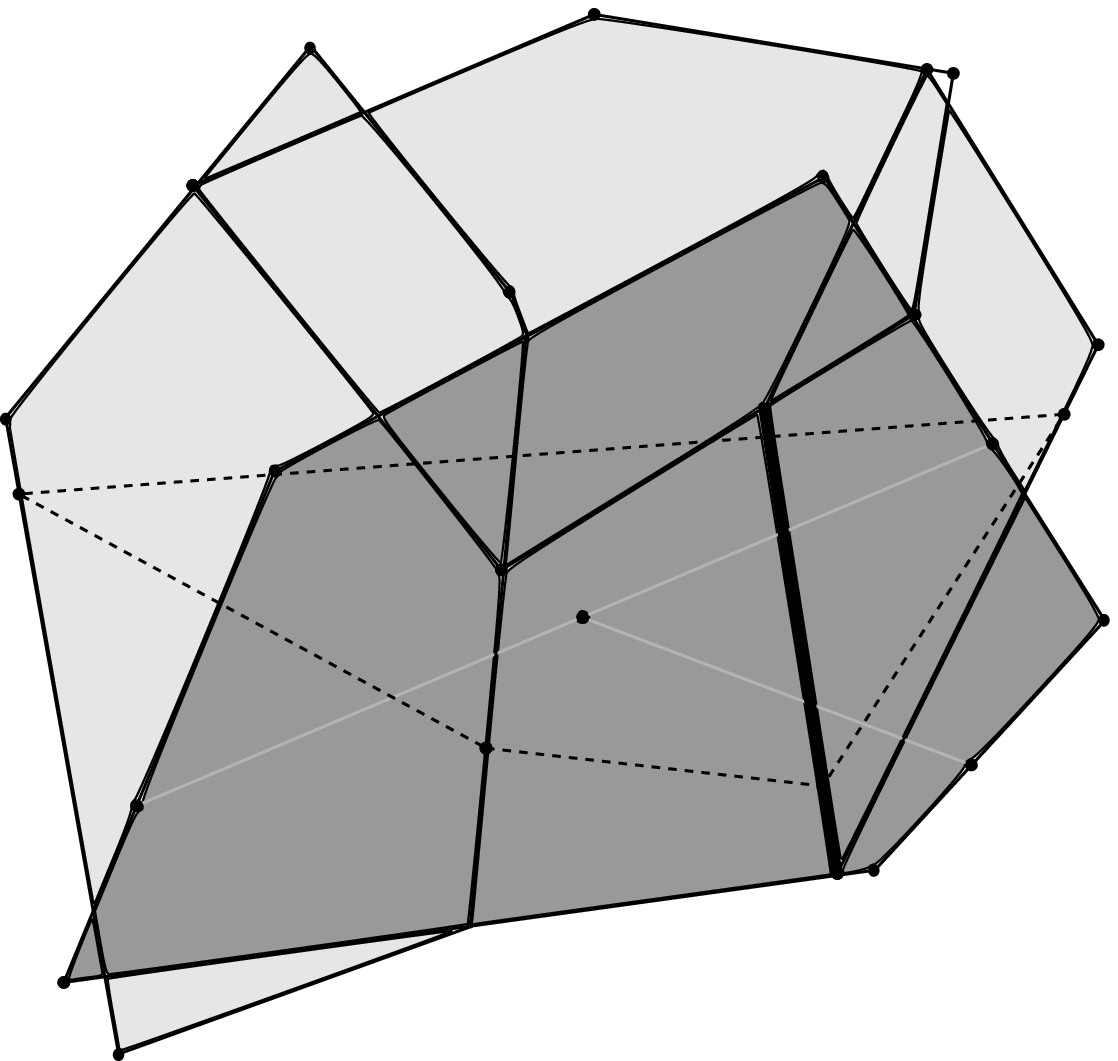}
\caption{An I-segment (bold) together with its carrying I-polygon (dark grey) at the time of its birth (left) and after further subdivision (right).}
\label{figpolygon}
\end{center}
\end{figure}

As already mentioned, related problems in the planar case have been studied in \cite{mnwmathnachr,thaele}. The results obtained there have led to a deeper understanding of planar STIT tessellations and were used to obtain new results about their structure. In a companion paper \cite{tw2} we also use the results of the present work to explore in more detail the combinatorial structure of spatial STIT tessellations, a study initiated in \cite{tw,wc}.

The paper is organized as follows: In Section \ref{secgeneral} we introduce some notation, recall the basic construction of STIT tessellations and rephrase some of their most important properties, which are needed for our later arguments. The conceptual framework as well as our main results are contained in Section \ref{secresults}. The proofs are the content of Section \ref{secproofs}.

\section{STIT tessellations in $\RR^d$}\label{secgeneral}

In this section we explain our basic notation, the construction and the main properties of STIT tessellations that are needed for our arguments below. Although our results are dealing with the three-dimensional case, we focus in this section on general space dimensions, because in our proofs we apply some of the properties also to lower-dimensional tessellations.

\subsection{Notation}\label{secnotation}

A \textit{tessellation} of $\RR^d$ with $d\geq 1$ can be described in two ways: as a locally finite collection of non-overlapping and space-filling compact convex polytopes (called \textit{cells} in the sequel), as well as a closed subset of $\RR^d$, which is formed by the union of all cell boundaries. We switch between both perspectives arbitrarily.

In this paper we will deal with \textit{random} tessellations of $\RR^d$, which can be regarded as random variables taking values in the measurable space of tessellations of $\RR^d$, see \cite{s/w,stkm} for a detailed definition and in particular for measurability issues. Our attention will be restricted to \textit{homogeneous} (spatially stationary) random tessellations, i.e. random tessellation whose distribution is invariant under spatial translations. 

We denote by ${\cal H}$ the set of all hyperplanes in ${\mathbb R}^d$. A hyperplane $h\in {\cal H}$ will be parametrized by its normal direction $u\in {\cal S}^{d-1}_+$ (the upper unit half-sphere in $\RR^d$) and its signed distance $p\in \RR$ to the origin, where the distance is defined as positive iff the projection of the origin $o$ onto $h$ is located in the upper half-space. Such a hyperplane is denoted by $h=h(p,u)\in {\cal H}$. For $B\subset {\mathbb R}^d$ define 
\begin{equation}
\nonumber [B]=\{ (p,u) \in \RR \times {\cal S}^{d-1}_+:\, h(p,u) \cap B\not= \emptyset \}
\end{equation} 
as the set of parameter values $(p,u)$ of hyperplanes $h(p,u)$ hitting the set $B$. Consider a measure $\Lambda$ on $\RR \times {\cal S}^{d-1}_+$ (equipped with the Borel product $\sigma$-field), which is the image under the described parametrization of a (non-zero) locally finite, translation invariant measure on ${\cal H}$. Invariance under translations implies that $\Lambda$ factorizes, i.e. there is a constant $\lambda>0$ and a probability measure ${\cal R}$ on  ${\cal S}^{d-1}_+$ with
\begin{equation}\label{Factorization}
\Lambda = \lambda \, \mu \otimes {\cal R},
\end{equation}
where $\mu$ denotes the Lebesgue measure on $\RR$. If $\cal R$ is the uniform distribution on ${\cal S}_+^{d-1}$ and $\lambda = 1$, then $\Lambda$ corresponds to the isometry-invariant hyperplane measure $\Lambda_{\rm iso}$, cf. \cite{s/w}. From now on we assume $\lambda = 1$ in the factorization (\ref{Factorization}). Moreover, to avoid degenerated cases we assume that $\cal R$ is not concentrated on a great half-subsphere of ${\cal S}_+^{d-1}$, i.e. we require $\text{span}(\text{support}({\cal R}))={\Bbb R}^d$. 

\subsection{Construction}\label{secconstructionstit}

A formal and detailed description of  STIT tessellations in bounded windows is given in \cite{nwstit}. Here we explain the construction in an intuitive way only. Let $\Lambda$ be a hyperplane measure as  above and $W\subset\RR^d$ a bounded convex polytope. We assign to $W$ a random lifetime and on expiry of this lifetime, we choose a random hyperplane, which splits $W$ into two polyhedral sub-cells $W^+$ and $W^-$. The construction continues now independently and recursively in both of the sub-cells $W^+$ and $W^-$, which is to say that $W^+$ and $W^-$ are provided with independent random lifetimes and that they are divided by random hyperplanes when they die. Note that the hyperplanes are always chopped-off by the boundary of their respective mother cells. This repeated cell division is continued until a fixed deterministic time threshold $t>0$ is reached. The random tessellation constructed until time $t$ within $W$ is denoted by $Y(t,W)$.

In order to ensure the temporal Markov property of the described cell splitting process, we assume that the lifetimes of the cells are conditionally (given the cells at a certain time) independent and exponentially distributed. Moreover, we assume that the parameter of this exponential lifetime distribution of a cell $c$ is given by $\Lambda([c])$ (and thus the parameters of the exponential distributions of different cells are not independent). In the special case $\Lambda=\Lambda_{\rm iso}$ we have that $\Lambda_{\rm iso}([c])$ is proportional to the integral-geometric mean width of $c$. Note that this choice  ensures that smaller cells live stochastically longer. Furthermore, we will assume that the hyperplane splitting a cell $c$ is chosen according to the law $\Lambda([c])^{-1}\Lambda(\cdot\cap [c])$. 

\subsection{Important properties}

We summarize here those properties of STIT tessellations that are needed in our arguments below. For further background on STIT tessellations we refer the reader to \cite{mnwconstr,nwstit} and to \cite{s/w,stkm} for a general introduction to stochastic geometry.

If we insert independent copies $Y_c(s)$ of the tessellation $Y(s)$ into the cells $c$ of $Y(t)$ (generating $Y_c(s)\cap c$), the resulting {\it iterated} or {\it nested tessellation} is denoted by $$Y(t) \boxplus Y(s)=Y(t)\cup\bigcup_{\text{$c$ a cell of $Y(t)$}}(Y_c(s)\cap c).$$ Considered as a random closed set, $Y(t,W)$ satisfies the following

\begin{description}
 \item \textbf{[Spatial consistency]} The random tessellation $Y(t,W)$ is spatially consistent in that for any convex $V\subset W$ with $W$ as in Subsection \ref{secnotation} we have $Y(t,W)\cap V\overset{D}{=}Y(t,V)$ (where $\overset{D}{=}$ stands for equality in distribution). Thus, Kolmogorov's extension theorem ensures that there exists a homogeneous random tessellation $Y(t)$ in the whole $\RR^d$ satisfying $Y(t)\cap W\overset{D}{=}Y(t,W)$. For the particular choice $\Lambda=\Lambda_{\rm iso}$, $Y(t)$ is also isotropic, i.e. its distribution is rotation invariant.
\end{description} 

In \cite{mnwconstr} a global construction of $Y(t)$ is provided, where the spatio-temporal random \textit{process} $(Y(t))_{t>0}$ is defined. Its important properties are summarized here:
\begin{description} 
 \item \textbf{[Scaling]}  The distributions of the rescaled tessellations are identical, i.e. $t\, Y(t)\overset{D}{=}s\,  Y(s)$ for all $s,t>0$.
 \item \textbf{[Iteration stability]} We have
 \begin{equation}\label{BOXPLUS}
   Y(t)\stackrel{D}{=}Y(s)\boxplus Y(t-s),\qquad 0<s<t,
 \end{equation}
 where $Y(s)$ and $Y(t-s)$ are independent. Consequently, because of \textbf{[Scaling]}, $$Y(t) \overset{D}{=} n Y(nt) \overset{D}{=} n(\underbrace{Y(t)\boxplus\ldots\boxplus Y(t)}_{n\ \text{times}})$$ for $t > 0$ and $n \in {\Bbb N}.$ The latter relation is usually referred to as {\it \underline{st}ability under \underline{it}erations}. For this reason, the random tessellations $Y(t)$ are called {\it STIT} tessellations.
  \item \textbf{[Poisson typical cell]} The distribution of the interior of the typical cell of $Y(t)$ coincides with that of the interior of the typical cell of a homogeneous Poisson hyperplane tessellation with intensity measure $t\Lambda$ (cf. \cite{s/w,stkm}).
  \item \textbf{[Interpretation of $t$ and ${\cal R}$]} The surface density $S_V$ of $Y(t)$, this is the mean total $(d-1)$-volume of cell boundaries of $Y(t)$ per unit $d$-volume, is equal to the construction time $t>0$, i.e. $S_V=t$. The probability measure ${\cal R}$, see (\ref{Factorization}), is the distribution of the normal direction in the typical boundary point of $Y(t)$, called the  \textit{directional distribution} of the tessellation (this is the surface-area-weighted directional distribution of the cell boundaries).
 \item \textbf{[Linear sections]} The intersection of $Y(t)$ with a line $L$ induces a homogeneous Poisson point process on $L$ with intensity $\Lambda([e(L)])t$, where $e(L)$ is a segment of unit length on $L$.
 \item \textbf{[Temporal Markov property]} The random process $(Y(t))_{t>0}$ with values in the space of  tessellations satisfies the Markov property in time. 
 
\end{description}


\section{Framework and results}\label{secresults}

In this section we present our main results in Theorem \ref{thm}, Theorem \ref{thm3}, Theorem \ref{thm2} and their corollaries. Before we can state them, we have to introduce further notation and the framework of our work.

\subsection{The marked process of I-segments in spatial STIT tessellations}

The notion of I-segments is due to R. Miles and characterizes one of the types of line segments associated with a random planar tessellation that is not face-to-face, cf. \cite{mm} and the references cited therein. An \textit{I-segment} can be defined as a maximal union of collinear and connected line segments that appear in the edge-skeleton of a tessellation. Obviously, this definition does not depend on the dimension of the ambient space and will henceforth also be used for the spatial case, $d=3$. Generalizing Miles' concept, an \textit{I-polygon} is a maximal union of coplanar and convex planar polygons (subsets of a plane in $\RR^3$). Here, `maximal' is understood in the sense that there is no extension to a larger collinear (coplanar) and convex set in the tessellation (considered as the closed set of cell boundaries).

In terms of the spatio-temporal construction of STIT tessellations $Y(t)$, the I-segments and I-polygons in spatial STIT tessellations can be characterized as follows. Any I-polygon of $Y(t)$ is a cell-dividing planar polygon introduced until time $t$, and any I-segment is a side ($1$-face) of an I-polygon. When a cell $c$ is divided by a plane $h\in[c]$, exactly one I-polygon $c\cap h$ is born, whereas at least three I-segments are born simultaneously. Any of these I-segments is the intersection of the I-polygon $c\cap h$ with a facet ($2$-face) of the cell $c$. This facet itself is embedded in a (possibly larger) previously born I-polygon, which is called the \textit{carrying I-polygon} of the I-segment.

For our purposes it will be appropriate to describe the joint distribution of the following marks associated to an I-segment in $Y(t)$:
\begin{equation}\label{defmarks}
(\ell, \varphi, \beta, \beta_{\rm carr})\in(0,\infty)\times{\cal S}_+^2\times(0,t)\times(0,t).
\end{equation}
Here $\ell,\varphi$ and $\beta$ are the length, direction and the birth time of the I-segment, respectively, and $\beta_{\rm carr}$ is the birth time of the carrying I-polygon. Note that the direction of a segment is the unique unit vector in ${\cal S}_+^2$ parallel to the segment.

Consider for fixed time $t>0$ and measure $\Lambda$ the homogeneous STIT tessellation $Y(t)$. Then the process of the \textit{marked} I-segments of $Y(t)$ with marks as in (\ref{defmarks}) is a homogeneous marked segment process. Thus, Palm calculus for marked point processes can be applied and allows to define the distribution of the \textit{typical I-segment} and its \textit{mark distribution}, see \cite{s/w,stkm} and also Section \ref{secproofs} below. In intuitive terms, the typical I-segment and its mark distribution can be regarded as a randomly (equally likely) chosen I-segment of $Y(t,W)$ together with its random marks when $W$ is a `large' observation window.

In the following we use the indicator function notation ${\bf 1}\{\ldots\}$, which is $1$ if the statement in brackets is fulfilled and $0$ otherwise. Let us also define the two constants
\begin{eqnarray}
\zeta_2 &:=& \int\limits\limits_{{\cal S}_+^2}\int\limits\limits_{{\cal S}_+^2}[u,v]\;{\cal R}(du){\cal R}(dv),\label{defzeta2}\\
\nonumber \zeta_3 &:=& \int\limits_{{\cal S}_+^2}\int\limits_{{\cal S}_+^2}\int\limits_{{\cal S}_+^2}[u,v,w]\;{\cal R}(du){\cal R}(dv){\cal R}(dw), 
\end{eqnarray}
where $[u,v]$ is the area of the parallelogram spanned by $u,v\in{\cal S}_+^2$ and where $[u,v,w]$ stands for the volume of the parallelepiped spanned by $u,v,w\in{\cal S}_+^2$ (interpreted as unit vectors in ${\Bbb R}^3$ having one endpoint at the origin). Note that in the isotropic case these constants are given by $\zeta_2={\pi\over 4}$ and $\zeta_3={\pi\over 8}$. We further introduce the following probability distributions on ${\cal S}_+^2$:
\begin{eqnarray}
\widetilde{{\cal R}}(U) &:=& {1\over\zeta_2}\int\limits_{{\cal S}_+^2}\int\limits_{{\cal S}_+^2}{\bf 1}\{u^\perp\cap v^\perp\cap{\cal S}_+^2\in U\}[u,v]\;{\cal R}(du){\cal R}(dv),\label{deftildeR}\\
{\cal R}_{\rm typ}(U) &:=& {\zeta_2\over\zeta_3}\int\limits_{{\cal S}_+^2}{\bf 1}\{u\in U\}\Lambda([u])\;\widetilde{{\cal R}}(du),\label{defrtyp}
\end{eqnarray}
where $U\subset{\cal S}_+^2$ is a Borel set and where $\Lambda([u])$ stands for the $\Lambda$-measure of the set of all planes hitting the unit line segment connecting $u\in{\cal S}_+^2\subset{\Bbb R}^3$ with the origin $o$, and $u^\perp$ denotes the orthogonal complement of $u\in {\cal S}_+^2$. For a Poisson plane tessellation with intensity measure $\Lambda$ the distributions $\widetilde{{\cal R}}$ and ${\cal R}_{\rm typ}$ are the length-weighted distribution of the direction of edges and the directional distribution of the typical edge, respectively, see \cite{HS3}.
In the particular isotropic case we have $\Lambda([u])={1\over 2}$ for any $u\in{\cal S}_+^2$ and thus, $\widetilde{\cal R}$ and ${\cal R}_{\rm typ}$ are the uniform distribution on ${\cal S}_+^2$.

\subsection{Statement of results}

We are now prepared to present the main results of this paper. We start with a description of the distribution of marks, see (\ref{defmarks}), associated with the typical I-segment:
\begin{theorem}\label{thm} Let $Y(t)$ be a homogeneous random STIT tessellation in ${\Bbb R}^3$ with measure $\Lambda$ as in (\ref{Factorization}).
\begin{description}
 \item[(i)] The distribution of the direction of the typical I-segment of $Y(t)$ equals ${\cal R}_{\rm typ}$ as defined in (\ref{defrtyp}).
 \item[(ii)] The joint birth time density of $(\beta, \beta_{\rm carr})$ of the typical I-segment of $Y(t)$ equals
$$p_{\beta,\beta_{\rm carr}}(s,r) = {3s\over t^3}{\bf 1}\{0<r<s<t\}.$$
 \item[(iii)] The conditional length density of the typical I-segment of the STIT tessellation $Y(t)$, given $(\varphi,\beta,\beta_{\rm carr})=(u,s,r)\in{\cal S}_+^2\times(0,t)^2$ with $0<r<s<t$, is
$$p_{\ell |{\varphi=u,\beta=s ,\beta_{\rm carr}=r }}(x) = \Lambda([u])s\ {\rm e}^{-\Lambda([u])s x}{\bf 1}\{x>0\} .$$
\end{description}
\end{theorem}

It is interesting to note that the conditional length density  in part (iii) of the previous theorem does not depend on $r$ and the joint distribution of the birth time vector $(\beta,\beta_{\rm carr})$ in part (ii) is independent of  $\Lambda$ (resp. $\cal R$).

Now we derive some consequences of Theorem \ref{thm}. At first, we calculate some marginal and conditional birth time densities:
\begin{corollary}\label{cor1} For the homogeneous random STIT tessellation $Y(t)$ we have
\begin{eqnarray}
\nonumber p_{\beta_{\rm carr}}(r ) &=& {3\over 2}{t^2-r ^2 \over t^3}{\bf 1}\{ 0<r<t\},\\
p_{\beta}(s) &=& {3s ^2\over t^3}{\bf 1}\{ 0<s<t\} .\label{margd}
\end{eqnarray}
Moreover,
\begin{eqnarray}
\nonumber p_{\beta |\beta_{\rm carr} =r}(s  ) &=& {2s \over t^2-r ^2} {\bf 1}\{ 0<r<s<t\},\\
\nonumber p_{\beta_{\rm carr}|\beta =s}(r  ) &=& {1\over s }{\bf 1}\{ 0<r<s\}, 
\end{eqnarray}
i.e. the conditional birth time distribution of the carrying I-polygon, given that the I-segment is born at time $0<s<t$, is the uniform distribution on $(0,s )$.
\end{corollary}
Using now Theorem \ref{thm} (ii) and (iii) we obtain
\begin{corollary}\label{cor2}
The direction $\varphi$ and the birth time pair $(\beta,\beta_{\rm carr})$ of the typical I-segment of $Y(t)$ are independent. Moreover, the joint conditional density of length and birth times, given its direction $u\in{\cal S}_+^2$ of the segment, is

\begin{equation}
\nonumber p_{\ell,\beta,\beta_{\rm carr}|\varphi=u}(x,s,r) = {3\Lambda([u])s^2\over t^3}{\rm e}^{-\Lambda([u])sx}\ {\bf 1}\{x>0\}{\bf 1}\{0<r<s<t\}.
\end{equation}
\end{corollary}
Integration yields now the marginal density for the length of the typical I-segment of the STIT tessellation $Y(t)$. This density is already known from \cite{ST4}, see also Lemma \ref{lemsurvivalfunctions} below:
\begin{corollary}\label{cor3}
The length density of the typical I-segment of $Y(t)$ equals
 $$\hspace{-3.5cm}p_{\ell}(x) =\int\limits_{{\cal S}_+^2}{3\over\Lambda([u])^3t^3x^4} \left(6-(6+6\Lambda([u])tx\right.$$ 
 $$\hspace{2.5cm}\left.  +3\Lambda([u])^2t^2x^2+\Lambda([u])^3t^3x^3)e^{-\Lambda([u])tx}\right){\cal R}_{\rm typ}(du),\ \ \ x>0.$$ 
In the isotropic case this reduces to $$p_{\ell}(x)={3\over t^3x^4}\left(48-(48+24tx+6x^2t^2+t^3x^3)e^{-{t\over 2}x}\right),\ \ \ x>0.$$
\end{corollary} 

Now we turn to the distribution of the number of vertices in the relative interior of the typical I-segment. In fact, with Theorem \ref{thm} we obtain the following:
\begin{theorem}\label{thm3}  The probability ${\mathsf p}_n$ that the typical I-segment of $Y(t)$ has exactly $n\in\NN$ vertices in its relative interior is given by \be{eqpn} {\mathsf p}_n=3\int\limits_0^1\int\limits_0^1(1-a)^3{(3-(1-a)(3-b))^n\over(3-(1-a)(2-b))^{n+1}}db\, da.\ee 
\end{theorem}
It is interesting to note that the distribution in Theorem \ref{thm3} does not depend on the measure $\Lambda$ (or equivalently the directional distribution $\cal R$) and the time parameter $t$. But, this is evident, because the number of vertices in the relative interior of the typical I-segment does not change when the tessellation is scaled in space. However, the latter is, because of \textbf{[Scaling]}, equivalent to a rescaling of time. Some particular values for ${\mathsf p}_n$ are summarized in the following table.
\begin{center}
\begin{tabular}[t]{|cc||lc|}
\hline
\parbox[0pt][2em][c]{0cm}{} & & ${\mathsf p}_n$ (exact value) & ${\mathsf p}_n$ (numerical value)\\
\hline
\parbox[0pt][2em][c]{0cm}{} & $n=0$ & ${189\over 8}\ln 3-26\ln 2-{15\over 2}$ & $0.43289$\\
\parbox[0pt][2em][c]{0cm}{} & $n=1$ & ${1593\over 16}\ln 3-107\ln 2-35$ & $0.21384$\\
\parbox[0pt][2em][c]{0cm}{} & $n=2$ & ${5319\over 16}\ln 3-350\ln 2-{245\over 2}$ & $0.11841$\\
\parbox[0pt][2em][c]{0cm}{} & $n=3$ & ${31617\over 32}\ln 3-1025\ln 2-{4499\over 12}$ & $0.07075$\\
\hline
\end{tabular}
\end{center}

From formula (\ref{eqpn}) in Theorem \ref{thm3} we conclude
\begin{corollary}\label{cor4} The mean number of vertices in the relative interior of the typical I-segment is equal to $2$. Moreover, the variance equals $59/3$ and all higher moments of that random variable are infinite.
\end{corollary}
Note that the mean value is in accordance with the result in \cite{tw}. Furthermore, the non-existence of higher moments is not surprising, because also moments of order $\geq 3$ of the length of the typical I-segment of $Y(t)$ are infinite, too.

\begin{figure}
\begin{center}
\includegraphics[width=0.9\columnwidth]{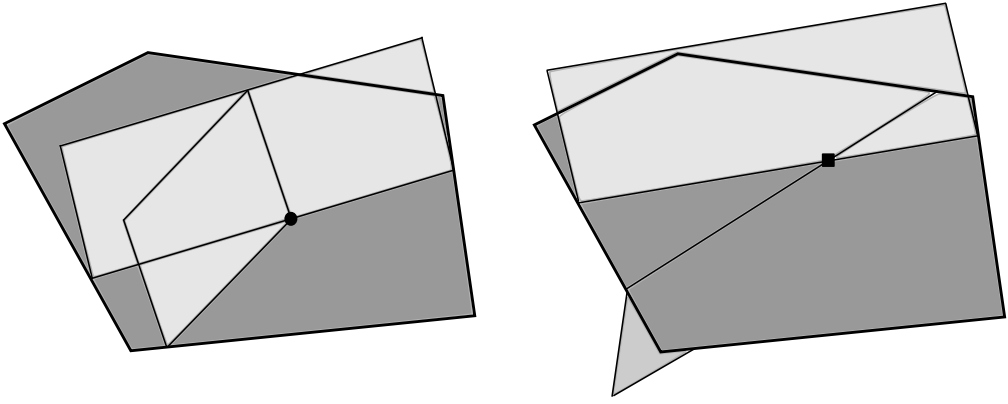}
\caption{A T-vertex (left) and an X-vertex (right) in a spatial STIT tessellation.}
\label{figknoten}
\end{center}
\end{figure}
As a last result we would like to point out that Theorem \ref{thm3} admits a refinement, which is needed in \cite{tw2}. We note at first that a spatial STIT tessellation has exactly two different types of vertices called T-vertices and X-vertices, see \cite{tw}-\cite{wc}. Illustrations of these two types of vertices are shown in Figure \ref{figknoten}. Given a carrying I-polygon (dark grey), a T-vertex is generated if two further I-polygons intersect in the \textit{same} half-space determined by the carrying I-polygon. An X-vertex is generated by an intersection of two further polygons in the two \textit{different} half-spaces specified by the carrying I-polygon.

For the typical I-segment we denote by ${\mathsf p}_{m,n}$, $m,n\in{\Bbb N}$, the probability that it has exactly $m$ vertices of type T and $n$ vertices of type X in its relative interior.
\begin{theorem}\label{thm2}
For a homogeneous spatial STIT tessellation and $m,n\in{\Bbb N}$ we have $${\mathsf p}_{m,n}=3\cdot 2^m{m+n\choose m}\int\limits_0^1\int\limits_0^1 (1-a)^3a^m{(1-(1-a)(1-b))^n\over(3-(1-a)(2-b))^{m+n+1}}db\,da.$$
\end{theorem}
From this  formula the following can be concluded:
\begin{corollary}\label{cor5}
The mean number of T-vertices and of X-vertices in the relative interior of the typical I-segment equals $1$, respectively. Furthermore, the variance of the number of T-type vertices is $8$ and that of the number of X-type vertices is $11/3$ (all higher moments are infinite). In addition, the covariance of the number of T- and X-vertices equals $8$.
\end{corollary}
In \cite{tw,wc} it is shown that the proportion of the intensities of T- and X-vertices in a spatial STIT tessellation is $2:1$.   Whereas any T-vertex is located in the relative interior of exactly one I-segment an X-vertex is in the relative interior of two I-segments. This confirms the first statement of Corollary \ref{cor5}.

\section{Proofs}\label{secproofs}

Before proving our results for the \textit{typical} I-segment, we first consider the mark distribution in a typical edge point, i.e. the \textit{length-weighted} mark distribution. The corresponding results are derived in Subsection \ref{seclengthweighted}. The proofs of Theorem \ref{thm}, Theorem \ref{thm3} and Theorem \ref{thm2} and their corollaries are the content of Subsection \ref{secprooftheorem}. Some preparatorial material is collected in Subsection \ref{secIsegmentdistributions}.

\subsection{Some preparations}\label{secIsegmentdistributions}

We denote by $\cal L$ the measurable space of line segments in ${\Bbb R}^3$, and equip it with the Borel $\sigma$-field induced by the Hausdorff distance, cf. \cite{s/w}. The distribution on $\cal L$ of the typical I-segment of a STIT tessellation $Y(t)$ for fixed $t>0$ and fixed measure $\Lambda$ is denoted by ${\Bbb D}^{Y(t)}$ and, similarly, the distribution of the typical edge of a Poisson plane tessellation with intensity measure $s\Lambda$ will be denoted by ${\Bbb D}^{P(s)}$.

In Theorem 3 in \cite{ST4} the following is shown:
\begin{lemma}\label{lemdistributiontypicalsegment}
For any non-negative measurable function $f:{\cal L}\rightarrow{\Bbb R}$ we have $$\int\limits f(L)\ {\Bbb D}^{Y(t)}(dL)=\int\limits_0^t\int\limits{3s^2\over t^3}f(L)\ {\Bbb D}^{P(s)}(dL)ds.$$
\end{lemma}
In order to apply this proposition we will need the length distribution of the edges in Poisson plane tessellations. 
We consider the survival function of the conditional length distribution of the typical edge of $P(s)$, given its direction $u\in{\cal S}_+^2$, i.e.
$$H_{\ell|\varphi=u}^{P(s)} (x) =\int\limits {\bf 1}\{\text{length}(L)>x\} \  {\Bbb D}^{P(s)}_{\varphi(L)=u}(dL)$$
where ${\Bbb D}^{P(s)}_{\varphi(L)=u}(dL)$ denotes the respective conditional distribution of the typical edge.
Analogously, the length-weighted case $\widetilde{H}_{\ell|\varphi=u}^{P(s)}$ is defined. In the following we will always use a tilde to indicate that we refer to a length-weighted distribution.

\begin{lemma}\label{lemsurvivalPoisson}
The conditional length distribution of the typical edge in the Poisson plane tessellation $P(s)$  has the survival function 
$$H_{\ell|\varphi=u}^{P(s)} (x) = e^{-\Lambda([u])sx}\ \ \ \text{for}\ \ \ x>0$$
and the survival function of the conditional length distribution of the length-weighted typical edge in the Poisson plane tessellation $P(s)$ equals 
$$\widetilde{H}_{\ell|\varphi=u}^{P(s)} (x) = (1+\Lambda([u])sx)e^{-\Lambda([u])sx}\ \ \ \text{for}\ \ \ x>0.$$ 
\end{lemma}
\begin{proof}
The intersection of the homogeneous Poisson plane process $P(s)$ with a fixed plane parallel to $u\in{\cal S}_+^2$ is a Poisson line process. The intersection of this line process with a fixed line parallel to $u$ is a Poisson point process with intensity $\Lambda([u])s$, see e.g. \cite[Theorem 4.4.6]{s/w} or, more specifically, (4.31) ibidem. Thus a twofold application of Slivnyak's theorem for Poisson processes yields that the typical edge point of a Poisson plane tessellation is a.s. located on a length-weighted segment  that is generated by a linear homogeneous Poisson point process, and, under the condition that the direction is $u$, its intensity is $\Lambda([u])s$. Hence the length of the typical edge, under the condition $\varphi =u$ is exponentially distributed with parameter $\Lambda([u])s$, and the respective length weighted distribution is the Gamma distribution with the parameter $(2,\Lambda([u])s)$.\hfill $\Box$
\end{proof}
As a corollary we obtain the  mean length of the typical edge in a Poisson plane tessellation.
\begin{lemma}\label{lemmeanlengthPoisson}
The  mean length  of the typical edge in the Poisson plane tessellation $P(s)$ is
$$\bar\ell^{P(s)}={1\over s}{\zeta_2\over\zeta_3}  .$$
\end{lemma}
\begin{proof}
With  Theorem 1 in \cite{HS3}, which says that ${\cal R}_{\rm typ}$ is the directional distribution of the typical edge in a Poisson plane tessellation, and (\ref{defrtyp}) we obtain
\begin{eqnarray}
\nonumber \bar\ell^{P(s)} &=& \int\limits_0^\infty H_{\ell}^{P(s)} (x) dx = \int\limits_0^\infty  \int\limits_{{\cal S}_+^2} H_{\ell|\varphi=u}^{P(s)} (x) {\cal R}_{\rm typ}(du)\, dx\\
\nonumber &=& {\zeta_2\over\zeta_3}\int\limits_{{\cal S}_+^2} \int\limits_0^\infty e^{-\Lambda([u])sx}  \Lambda([u])  dx \, \widetilde{{\cal R}}(du) = {1\over s}{\zeta_2\over\zeta_3},
\end{eqnarray}
which proves our claim.\hfill $\Box$
\end{proof}

We derive now from Lemma \ref{lemdistributiontypicalsegment} a similar representation for the distribution $\widetilde{{\Bbb D}}^{Y(t)}$ of the length-weighted typical I-segment of $Y(t)$ in terms of $\widetilde{{\Bbb D}}^{P(s)}$, the distribution of the length-weighted typical edge of a Poisson plane tessellation. 
\begin{lemma}\label{lemdistributiontypicalsegmentweighted}
For any non-negative measurable function $f:{\cal L}\rightarrow{\Bbb R}$ we have $$\int\limits f(L)\ \widetilde{{\Bbb D}}^{Y(t)}(dL)=\int\limits_0^t\int\limits{2s\over t^2}f(L)\ \widetilde{{\Bbb D}}^{P(s)}(dL)ds.$$
\end{lemma}
\begin{proof} Denoting the mean length of the typical I-segment of $Y(t)$ by $\bar\ell^t$ and that of the typical edge of a Poisson plane tessellation $P(s)$ by $\bar\ell^{P(s)}$ we have
\begin{eqnarray}
\nonumber \int\limits f(L)\widetilde{{\Bbb D}}^{Y(t)}(dL) &=& {1\over\bar\ell^t}\int\limits f(L)\text{length}(L)\ {\Bbb D}^{Y(t)}(dL)\\
\nonumber &=&  {1\over\bar\ell^t}\int\limits_0^t\int\limits {3s^2\over t^3}f(L)\text{length}(L)\ {\Bbb D}^{P(s)}(dL)ds\\
\nonumber &=& {1\over\bar\ell^t}\int\limits_0^t{3s^2\over t^3}\bar\ell^{P(s)}\left(\int\limits f(L)\ \widetilde{{\Bbb D}}^{P(s)}(dL)\right)ds
\end{eqnarray}
with Lemma \ref{lemdistributiontypicalsegment}. Using now $\bar\ell^t={3\over 2t}{\zeta_2\over\zeta_3}$ from \cite{tw} and Lemma \ref{lemmeanlengthPoisson}, we conclude the assertion of the lemma.
\hfill $\Box$
\end{proof}

Now we consider the survival function of the conditional length distribution of the typical I-segment of $Y(t)$, given its direction $u\in{\cal S}_+^2$, i.e.
$$H_{\ell|\varphi=u}^{t} (x) =\int\limits {\bf 1}\{\text{length}(L)>x\} \  {\Bbb D}^{Y(t)}_{\varphi(L)=u}(dL)$$
where ${\Bbb D}^{Y(t)}_{\varphi(L)=u}(dL)$ denotes the respective conditional distribution of the typical I-segment.
Analogously, the length-weighted case $\widetilde{H}_{\ell|\varphi=u}^{t}$ is defined.
The two distributional identities in Lemma \ref{lemdistributiontypicalsegment} and Lemma \ref{lemdistributiontypicalsegmentweighted} together with Lemma \ref{lemsurvivalPoisson} imply the following result.
\begin{lemma}\label{lemsurvivalfunctions} The length distribution of the typical I-segment of $Y(t)$, given its direction $u\in{\cal S}_+^2$, has the survival function 
$$H_{\ell|\varphi=u}^{t} (x) =\int\limits_0^t{3s^2\over t^3}e^{-\Lambda([u])sx}ds$$ 
and the length distribution of the length-weighted typical I-segment of $Y(t)$, given its direction $u\in{\cal S}_+^2$, has the survival function 
$$\widetilde{H}_{\ell|\varphi=u}^{t} (x) =\int\limits_0^t{2s\over t^2}(1+\Lambda([u])sx)e^{-\Lambda([u])sx}ds.$$
\end{lemma}

\subsection{Length-weighted mark distributions}\label{seclengthweighted}

For construction times $0<r<s$ we consider the states $Y(r)$ and $Y(s)$ and introduce three different random sets of unions of I-segments of $Y(s)$. These are:
\begin{itemize}
\item[(a)] the union of all I-segments of $Y(s)$,
\item[(b)] the union of all those I-segments of $Y(s)$ with birth time in $(r,s]$ that appear {\em in the interior} of the cells of $Y(r)$,
\item[(c)] the union of I-segments of $Y(s)$ with birth time in $(r,s]$ that appear {\em on the facets} of the cells of $Y(r)$.
\end{itemize}
For these random sets the corresponding length measures, the length intensities and the directional distributions in the typical edge point (the length-weighted case) are, respectively, denoted by
\begin{itemize}
\item[(a)]  $\mu(s,.)$, $L_V(s)={\Bbb E}\mu(s,[0,1]^3)$ and $\widetilde{\mathbb Q}^s_\varphi$,
\item[(b)] $\mu^*(s-r,.)$, $ L^*_V(s-r)={\Bbb E}\mu^*(s-r,[0,1]^3)$ and  $\widetilde{\mathbb Q}^{*,s-r}_\varphi$,
\item[(c)] $\mu(r,s,.)$, $L_V (r,s)={\Bbb E}\mu(r,s,[0,1]^3)$ and $\widetilde{\mathbb Q}^{r,s}_\varphi$.
\end{itemize}
For a Borel set $B\subset{\Bbb R}^3$, e.g. $\mu^*(s-r,B)$ is the total length of edges of $Y(s)$ with birth time in $(r,s]$, which are located in the interior of the cells of $Y(r)$ and in $B$.
It is evident that, for $0<r<s$, we have for the length measures 
\begin{equation}\label{lengthmeasure}
\mu (s,\cdot)= \mu (r,\cdot) + \mu (r,s,\cdot ) + \mu^* (s-r,\cdot ).
\end{equation}
From the STIT property the following relations can be deduced:
\begin{lemma}\label{neueslemma} We have $L_V^*(s-r)=L_V(s-r)$ and $L_V(s)=\zeta_2s^2$ with $\zeta_2$ given by (\ref{defzeta2}). Furthermore, $\widetilde{\mathbb Q}^{*,s-r}_\varphi =\widetilde{\mathbb Q}^{s-r}_\varphi=\widetilde{{\Bbb Q}}_\varphi^s=\widetilde{{\cal R}}$ with $\widetilde{\cal R}$ defined via (\ref{deftildeR}).
\end{lemma}
\begin{proof} Resorting to (\ref{BOXPLUS}) we find $L^*_V (s-r)=L_V(s-r)$ and
$\widetilde{\Bbb Q}^{*,s-r}_\varphi =\widetilde{\Bbb Q}^{s-r}_\varphi$. Equation $L_V(s)=\zeta_2s^2$ is (8) in \cite{nw3d} and the fact that $\widetilde{{\Bbb Q}}_\varphi^s$ does not depend on $s$ follows from \cite[Eq. (14)]{nw3d}. The remaining equality $\widetilde{{\Bbb Q}}_\varphi^s=\widetilde{{\cal R}}$ is a consequence of Eq. (11) ibidem.\hfill $\Box$
\end{proof}
In a next step we consider for the STIT tessellation $Y(t)$ with $t>0$ the joint length-weighted distribution $\widetilde{\mathbb Q}^t_{\varphi , \beta , \beta_{\rm carr}}$ of direction $\varphi\in{\cal S}_+^2$ and birth times $\beta,\beta_{\rm carr}\in(0,t]$.
\begin{lemma}\label{lengthweibirthdir}
For any Borel set $U\subset{\cal S}_+^{2}$ and $0<r<s<t$ we have
$$\widetilde{\mathbb Q}^t_{\varphi , \beta , \beta_{\rm carr}} (U\times (r,s] \times (0,r])= \frac{2(rs-r^2)}{t^2} \widetilde{\cal R}(U).$$ The corresponding joint density of $(\beta , \beta_{\rm carr})$ with respect to the Lebesgue measure on $(0,t)^2$ is
$$
\widetilde p_{ \beta, \beta_{\rm carr}} (s,r) =  \frac{2}{t^2}{\bf 1}\{0<r<s<t\},
$$
i.e. $\varphi$ and $( \beta , \beta_{\rm carr})$ are independent and $( \beta , \beta_{\rm carr})$ is uniformly distributed on the triangle $\{(s,r)\in {\mathbb R}^2:\, 0<r<s<t\}$.
\end{lemma}
\begin{proof} We use (\ref{lengthmeasure}), to conclude
\begin{equation}\label{lengthint}
L_V(s)=L_V(r) + L^*_V (s-r) + L_V (r,s)
\end{equation}
for the length intensities defined above. In view of Lemma \ref{neueslemma} we obtain from (\ref{lengthint}), \begin{equation}L_V (r,s) = 2\zeta_2 (rs-r^2).\label{gleichungfuerlvrs}\end{equation} From the definition of the length-weighted directional distribution it follows that the mean total length per unit volume of I-segments of $Y(t)$ with direction in a Borel set $U\subset{\cal S}_+^{2}$ is $L_V(t)\widetilde{{\cal R}}(U)$. Combining this with (\ref{lengthint}) and (\ref{gleichungfuerlvrs}) yields $$L_V(r,s)\widetilde{{\cal R}}(U) = 2\zeta_2(rs-r^2)\widetilde{{\cal R}}(U).$$ Thus, for the  mark distribution in a typical edge point we find
$$\widetilde{\mathbb Q}^t_{\varphi , \beta , \beta_{\rm carr}}( U \times (r,s] \times  (0,r])=\frac{L_V(r,s)\widetilde{{\cal R}}(U)}{L_V(t)} =\frac{2(rs-r^2)}{t^2}\widetilde{{\cal R}}(U)$$ and hence, $\varphi$ and $(\beta, \beta_{\rm carr})$ are independent. Since $\beta_{\rm carr} < \beta $ with probability 1 we have for the joint distribution function $\widetilde F_{\beta, \beta_{\rm carr}}(s,r)$ of $\beta$ and $\beta_{\rm carr}$ with $0<r<s<t$  the equation
\begin{equation}\label{gleichungfuerqtbetabetacarr}
\widetilde{\mathbb Q}^t_{ \beta , \beta_{\rm carr}}(  (r,s]\times (0,r] )= \widetilde F_{\beta, \beta_{\rm carr}} (s,r) - \widetilde{\mathbb Q}^t_{ \beta , \beta_{\rm carr}}( (0,r]\times (0,r] ),
\end{equation}
where $\widetilde{\mathbb Q}^t_{ \beta , \beta_{\rm carr}}((r,s]\times(0,r])=\widetilde{{\Bbb Q}}^t_{\varphi,\beta,\beta_{\rm carr}}({\cal S}_+^2\times(r,s]\times(0,r])$ stands for joint distribution of $(\beta,\beta_{\rm carr})$ in a typical edge point of $Y(t)$. Partial differentiation with respect to $r$ and then with respect to $s$ yields $\widetilde p_{ \beta, \beta_{\rm carr}} (s,r)$ (note that the $\widetilde{\Bbb Q}$-term on the right hand side of (\ref{gleichungfuerqtbetabetacarr}) does not depend on $s$ and thus vanishes after differentiation w.r.t. $s$) and completes the proof of the lemma.\hfill $\Box$
\end{proof}
We turn now to the joint length-weighted distribution $\widetilde{\mathbb Q}^t_{\ell, \varphi , \beta , \beta_{\rm carr}}$ of length, direction and birth times of the I-segment, i.e. of the I-segment through a typical edge point of $Y(t)$. The key is the conditional length distribution of the I-segment, given its direction $u \in {\cal S}_+^{2}$. For any point $z$ in the edge-skeleton of $Y(t)$ denote by $\ell (z), \varphi (z), \beta (z) , \beta_{\rm carr} (z)$ the a.s. uniquely determined mark of the I-segment through $z$. According to the definition of the mark distribution (see p.84 in \cite{s/w}) we have for any Borel set $U\subset {\cal S}_+^{2}$ using the Campbell theorem and (\ref{lengthmeasure})
\begin{eqnarray}
\nonumber & & \widetilde{\mathbb Q}^t_{\ell, \varphi , \beta , \beta_{\rm carr}} ( (x,\infty )\times  U \times (r,s] \times  (0,r]) \\
\nonumber &=& \frac{1}{L_V(t)}  {\mathbb E} \int\limits  {\bf 1}\{ \ell (z)>x, \varphi (z)\in U \}   \mu (r,s,dz)\\
\nonumber &=& \frac{1}{L_V(t)} {\mathbb E} \left[ \int\limits  {\bf 1}\{ \ell (z)>x, \varphi (z)\in U \}   \mu (s,dz)
                                        - \int\limits  {\bf 1}\{ \ell (z)>x, \varphi (z)\in U \}   \mu (r,dz)\right.\\
\nonumber & & \hspace{5cm} \left. -\int\limits  {\bf 1}\{ \ell (z)>x, \varphi (z)\in U \}   \mu^* (s-r,dz)\right] \\
\nonumber &=& \frac{1}{L_V(t)} \left[ L_V(s) \widetilde{\mathbb Q}^s_{\ell, \varphi } ( (x,\infty )\times  U)
                            - L_V(r) \widetilde{\mathbb Q}^r_{\ell, \varphi } ( (x,\infty )\times  U)\right.\\
& & \hspace{5cm} \left.- L_V(s-r) \widetilde{\mathbb Q}^{*,s-r}_{\ell, \varphi } ( (x,\infty )\times  U)\right]\label{laengwicht1},
\end{eqnarray}
where $\widetilde{{\mathbb Q}}^{*,s-r}_{\ell, \varphi }$ is the $(\ell,\varphi)$-marginal distribution of $\widetilde{\mathbb Q}^{*,s-r}_{\ell, \varphi , \beta , \beta_{\rm carr}}$. Write now $\widetilde{\mathbb Q}^s_{\ell | \varphi =u }$ for the length-weighted conditional distribution of length under the condition $\varphi =u$. Then
\begin{eqnarray}
\nonumber \widetilde{\mathbb Q}^s_{\ell, \varphi } ( (x,\infty )\times  U) &=&  \int\limits {\bf 1}\{l>x, u\in U \} \widetilde{\mathbb Q}^s_{\ell, \varphi } (d(l, u )) \\
 &=&  \int\limits \int\limits {\bf 1}\{ l>x \} \widetilde{\mathbb Q}^s_{\ell | \varphi =u } (dl) {\bf 1}\{ u\in U \} \widetilde{{\cal R}}(du),\label{conditionallength}
\end{eqnarray}
where we have used Lemma \ref{neueslemma}. In view of (\ref{laengwicht1}) and (\ref{conditionallength}) we now have to determine certain conditional length distributions.
\begin{lemma}\label{lemmaterm12}
For $0<r<s$, $0<x$ and a Borel set $U\subset{\cal S}_+^{2}$ we have 
\begin{eqnarray}
\nonumber  & & L_V(s) \widetilde{\mathbb Q}^s_{\ell, \varphi } ( (x,\infty )\times  U) - L_V(r) \widetilde{\mathbb Q}^r_{\ell, \varphi } ( (x,\infty )\times  U) \\
\nonumber  & = & \zeta_2\int\limits_U{2\over\Lambda([u])^2x^2} \left( [3+\Lambda([u])rx(3+\Lambda([u])rx)]e^{-\Lambda([u])rx} \right. \\
\nonumber  & & \left. \qquad \ -[3+\Lambda([u])sx(3+\Lambda([u])sx)]e^{-\Lambda([u])sx}\right)\widetilde{{\cal R}}(du).
\end{eqnarray}
\end{lemma}
\begin{proof}
We start by recalling from Lemma \ref{neueslemma} that $L_V(s)=\zeta_2s^2$ and $L_V(r)=\zeta_2r^2$. Using now (\ref{conditionallength}) and the integral representation for the conditional survival function provided in Lemma \ref{lemsurvivalfunctions} we calculate
\begin{eqnarray}
\nonumber & & L_V(s) \widetilde{\mathbb Q}^s_{\ell, \varphi } ( (x,\infty )\times  U) - L_V(r) \widetilde{\mathbb Q}^r_{\ell, \varphi } ( (x,\infty )\times  U)\\
\nonumber &=& \zeta_2s^2\int\limits_U\widetilde{H}_{\ell|\varphi=u}^{s}(x) \widetilde{{\cal R}}(du)-\zeta_2r^2\int\limits_U\widetilde{H}_{\ell|\varphi=u}^{r}(x) \widetilde{{\cal R}}(du)\\
\nonumber &=& 2\zeta_2\int\limits_U\int\limits_r^sv(1+\Lambda([u])vx)e^{-\Lambda([u])vx}dv\widetilde{{\cal R}}(du)
\end{eqnarray}
and integration proves the claim.\hfill $\Box$
\end{proof}
It remains to determine the last item $\widetilde{\mathbb Q}^{*,s-r}_{\ell, \varphi } ( (x,\infty )\times  U)$ in (\ref{laengwicht1}), which in view of Lemma \ref{neueslemma} may be written in the form \begin{equation}\widetilde{{\Bbb Q}}_{\ell,\varphi}^{*,s-r}((x,\infty)\times U)=\int\limits\int\limits{\bf 1}\{l>x\}\widetilde{{\Bbb Q}}_{\ell|\varphi=u}^{*,s-r}(dl){\bf 1}\{u\in U\}\widetilde{R}(du).\label{eqconditiondistr}\end{equation}
\begin{lemma}\label{lemmaterm3} For $0<r<s$, $0<x$ and a Borel set $U\subset{\cal S}_+^{2}$ we have
$$L_V(s-r)\widetilde{\mathbb Q}^{*,s-r}_{\ell,\varphi}((x,\infty )\times U)=\zeta_2\int\limits_U{2\over\Lambda([u])^2x^2}$$ $$\times\left(3(e^{-\Lambda([u])rx}-e^{-\Lambda([u])sx})+\Lambda([u])^2sx^2(re^{-\Lambda([u])rx}-se^{-\Lambda([u])sx})\right)\widetilde{{\cal R}}(du).$$
\end{lemma}
\begin{proof} To derive the result we make use of a method developed for the planar case in \cite{erev} and consider the conditional survival function $\widetilde{G}^{*,s-r}_{\ell|\varphi =u}$ of the length of the {\em remaining} I-segment, this is the part of the I-segment that lies {\em above} a random point of the edge skeleton, which is selected according to $\mu^* (s-r,\cdot)$. (In this proof, the letter $G$ will always refer to a remaining I-segment while $H$ refers to the whole I-segment.) Regarding (\ref{BOXPLUS}), the distribution of remaining length is the distribution of the minimum of the remaining length in $Y(s-r)$ and the distance (in direction $u$) to the `frame' tessellation $Y(r)$, where the I-segment is cut. Because of \textbf{[Linear sections]} the latter has the conditional survival function $e^{-\Lambda([u])rx}$. Making use of the independence of $Y(r)$ and $Y(s-r)$ we obtain
\begin{equation}\label{remainingsegmentsurvivaldef}
\widetilde{G}^{*,s-r}_{\ell | \varphi =u }(x) = \widetilde{G}^{s-r}_{\ell | \varphi =u }(x)\cdot e^{-\Lambda([u])rx},
\end{equation}
where $\widetilde{G}^{s-r}_{\ell | \varphi =u }$ is the corresponding conditional survival function for the remaining I-segment of the tessellation $Y(s-r)$. Palm theory yields (compare with Eq. (14) in \cite{erev}) that the survival function of the length of the remaining I-segment as considered above and the conditional survival function $H_{\ell|\varphi=u}^{s-r}$ of the length of the (whole) typical I-segment are related by
\begin{equation}\label{eqpalm2}
\widetilde{G}^{s-r}_{\ell | \varphi =u }(x)= {1\over\bar{\ell}^{s-r}_{\varphi =u}}\int\limits_x^\infty H_{\ell|\varphi=u}^{s-r}(a)da,
\end{equation}
where $\bar{\ell}^{s-r}_{\varphi =u}$ is the conditional mean length of the typical I-segment in $Y(s-r)$. Using Lemma \ref{lemsurvivalfunctions} we calculate $\bar{\ell}^{s-r}_{\varphi=u}={3\over 2\Lambda([u])(s-r)}$. Combining (\ref{remainingsegmentsurvivaldef}) with (\ref{eqpalm2}) and using again Lemma \ref{lemsurvivalfunctions} we find
\begin{equation}\widetilde{G}^{*,s-r}_{\ell | \varphi =u }(x)={2\Lambda([u])\over(s-r)^2}e^{-\Lambda([u])rx}\int\limits_x^\infty\int\limits_0^{s-r}v^2e^{-\Lambda([u])va}dvda.\label{eqpalm3}\end{equation} We use now once more the Palm theory developed in \cite{erev} (see in particular Eq. (16) there) to conclude that the survival function $H_{\ell|\varphi=u}^{*,s-r}$ of conditional length distribution of the corresponding typical I-segment equals 
$$H_{\ell|\varphi=u}^{*,s-r}(x) = \left(\lim_{x\downarrow 0}{\partial\widetilde{G}^{*,s-r}_{\ell | \varphi =u }(x)\over\partial x}\right)^{-1}{\partial\widetilde{G}^{*,s-r}_{\ell | \varphi =u }(x)\over\partial x}$$ 
and thus we find 
$$\hspace{-2cm}H_{\ell|\varphi=u}^{*,s-r}(x)={6\over\Lambda([u])^3(s-r)^2(2s+r)x^3}\left((2+\Lambda([u])rx)e^{-\Lambda([u])rx}\right.$$ $$\hspace{4cm}\left.-(2+(2s-r)\Lambda([u])x+\Lambda([u])^2s(s-r)x^2)e^{-\Lambda([u])sx}\right)$$
from (\ref{eqpalm3}) by integration. Using this formula we calculate the corresponding mean I-segment length $\bar\ell_{\varphi=u}^{*,s-r}$ as 
$$\bar\ell_{\varphi=u}^{*,s-r}=\int\limits_0^\infty H_{\ell|\varphi=u}^{*,s-r}(x)dx={3\over\Lambda([u])(2s+r)}$$ 
and again by length-weighting we get the corresponding survival function 
$$\widetilde{H}_{\ell|\varphi=u}^{*,s-r}(x)=\widetilde{\mathbb Q}^{*,s-r}_{\ell|\varphi=u}((x,\infty ))=\int\limits_x^\infty{z\over\bar\ell_{\varphi=u}^{*,s-r}}{\partial H_{\ell|\varphi=u}^{*,s-r}(z)\over\partial z}dz,$$ 
which is given by 
$$\hspace{-2cm}\widetilde{H}_{\ell|\varphi=u}^{*,s-r}(x)={2\over\Lambda([u])^2(s-r)^2x^2}\left(3(e^{-\Lambda([u])rx}-e^{-\Lambda([u])sx})+\right.$$ $$\hspace{5cm}\left.\Lambda([u])^2sx^2(re^{-\Lambda([u])rx}-se^{-\Lambda([u])sx})\right).$$ 
Taking into account the  factor $L_V(s-r)=\zeta_2(s-r)^2$, we can complete the proof by using (\ref{eqconditiondistr}).\hfill $\Box$
\end{proof}
Combining now (\ref{laengwicht1}) with Lemma \ref{lemmaterm12}, Lemma \ref{lemmaterm3} and the fact $L_V(t)=\zeta_2t^2$ from Lemma \ref{neueslemma} we arrive at 
$$\widetilde{\mathbb Q}^t_{\ell, \varphi , \beta , \beta_{\rm carr}} ( (x,\infty )\times  U \times (r,s] \times  (0,r])$$ $$=\int\limits_U{2r\over\Lambda([u])xt^2}\left((2+\Lambda([u])rx)e^{-\Lambda([u])rx}-(2+\Lambda([u])sx)e^{-\Lambda([u])sx}\right)\widetilde{{\cal R}}(du) .$$
A relation similar to (\ref{gleichungfuerqtbetabetacarr}) for $\widetilde{\mathbb Q}^t_{\ell, \varphi , \beta , \beta_{\rm carr}}$ and the respective distribution function yields after differentiation w.r.t. $r$, $s$ and $x$ the following assertion:
\begin{lemma}\label{neueskorollar}
The conditional joint density of length $\ell$ and the birth time vector $(\beta,\beta_{\rm carr})$ of the length-weighted typical I-segment in $Y(t)$, given its direction $\varphi=u\in{\cal S}_+^2$, equals $$\widetilde{p}_{\ell,\beta,\beta_{\rm carr}|\varphi=u}^t(x,s,r)={2\Lambda([u])^2s^2\over t^2}xe^{-\Lambda([u])sx}\, {\bf 1}\{ x>0\} {\bf 1}\{0<r<s<t\} .$$ 
\end{lemma}
We turn now to the proofs of our main results. In fact, Lemma \ref{neueskorollar} is the key for the proof of Theorem \ref{thm}, which forms the basis to derive Theorems \ref{thm3} and \ref{thm2}.

\subsection{Proof of the main results}\label{secprooftheorem}

\begin{proof}[Proof of Theorem \ref{thm}]  \textit{(i)} Apply Lemma \ref{lemdistributiontypicalsegment} with $f(L)={\bf 1}\{ \varphi (L)\in U\}$ for a measurable $U\subset{\cal S}_+^2$ and where $\varphi (L)$ the direction of $L$. Since for homogeneous Poisson plane tessellations with intensity measure $s\Lambda$, $s>0$, the directional distribution of the typical edge is invariant w.r.t. the scaling factor $s$, this implies that the distribution of the direction of the typical I-segment of $Y(t)$ is the same as the directional distribution of the typical edge in a Poisson plane tessellation with intensity measure $t\Lambda$. Thus we can apply Theorem 1 in \cite{HS3}, which immediately yields assertion (i).

\textit{(ii)} The key is the relation 
$$p_{\ell,\beta,\beta_{\rm carr}|\varphi=u}(x,s,r)=\frac{\bar{\ell}^{t}_{\varphi =u}}{x}\widetilde{p}_{\ell,\beta,\beta_{\rm carr}|\varphi=u}(x,s,r),$$ 
where $\bar{\ell}^{t}_{\varphi =u}$ is the conditional mean length of the typical I-segment in $Y(t)$, given $\varphi=u$. With Lemma \ref{lemsurvivalfunctions} we calculate $\bar{\ell}^{t}_{\varphi =u}={3\over 2t\Lambda([u])}$, and Lemma \ref{neueskorollar} implies 
\begin{eqnarray}
\nonumber p_{\ell,\beta,\beta_{\rm carr}|\varphi=u}(x,s,r) &=& {3\over 2\Lambda([u])tx}{2\Lambda([u])^2s^2\over t^2}xe^{-\Lambda([u])x} \, {\bf 1}\{ x>0\}\\
&=& {3s\over t^3}\Lambda([u])se^{-\Lambda([u])sx} \, {\bf 1}\{ x>0\}.\label{neueslabel1}
\end{eqnarray}
Integration with respect to $x$ yields now 
\begin{equation}\label{neueslabel2}
p_{\beta,\beta_{\rm carr}|\varphi=u}(s,r)={3s\over t^3}\int\limits_0^\infty\Lambda([u])se^{-\Lambda([u])sx}dx={3s\over t^3},\ \ \ 0<r<s<t.
\end{equation}
Since it does not depend on $u$, we have $p_{\beta,\beta_{\rm carr}}=p_{\beta,\beta_{\rm carr}|\varphi=u}$ and this completes the argument.

\textit{(iii)} 
Similarly to the proof of part (ii) above, we use the following relation for the conditional densities:
$$p_{\ell |\varphi=u,\beta=s,\beta_{\rm carr}=r}(x)=\frac{\bar{\ell}^{t}_{\varphi =u,\beta=s,\beta_{\rm carr}=r}}{x}\widetilde{p}_{\ell |\varphi=u,\beta=s,\beta_{\rm carr}=r}(x),$$ where $\bar{\ell}^{t}_{\varphi =u,\beta=s,\beta_{\rm carr}=r}$ is the conditional mean length of the typical I-segment in $Y(t)$, given $\varphi$, $\beta$ and $\beta_{\rm carr}$. Combining (\ref{neueslabel1}) and (\ref{neueslabel2}) we calculate $$p_{\ell|\varphi=u,\beta=s,\beta_{\rm carr}=r}(x)={p_{\ell,\beta,\beta_{\rm carr}|\varphi=u}(x,s,r)\over p_{\beta,\beta_{\rm carr}|\varphi=u}(s,r)}=\Lambda([u])se^{-\Lambda([u])sx}{\bf 1}\{x>0\},$$ which implies $\bar{\ell}^{t}_{\varphi =u,\beta=s,\beta_{\rm carr}=r}={1\over\Lambda([u])s}$. This completes the argument.\hfill $\Box$
\end{proof}

\begin{proof}[Proof of Corollaries \ref{cor1}--\ref{cor3}]
The formulas in Corollary \ref{cor1} follow by straightforward integration from Theorem \ref{thm} (ii).

To see Corollary \ref{cor2} we first note that in the proof of Theorem \ref{thm} (ii) we have already shown that 
$$p_{\ell,\beta,\beta_{\rm carr}|\varphi=u}(x,s,r)={3s\over t^3}\Lambda([u])se^{-\Lambda([u])sx},$$ 
which proves the formula. Integration with respect to $x$ yields now the conditional joint density of the birth time vector, namely $$p_{\beta,\beta_{\rm carr}|\varphi=u}(s,r)={3s\over t^3}.$$ This shows the independence of $\varphi$ and the birth times $(\beta,\beta_{\rm carr})$ and finally completes the proof of Corollary \ref{cor2}.

The statement of Corollary \ref{cor3} follows by  integration from the formula (\ref{neueslabel1}).\hfill $\Box$
\end{proof}

\begin{proof}[Proof of Theorem \ref{thm3}]
The first step is to verify the following expression for ${\mathsf p}_n$:
\begin{equation}\label{helpeq}
\begin{split}
{\mathsf p}_n= &\int\limits_0^{t}\int\limits_0^s\int\limits_{{\cal S}_+^2}\int\limits_0^\infty{3\Lambda([u])s^2\over t^3} e^{-s\Lambda([u])x}\\
&\qquad\times{(\Lambda([u])x(3t-2s-r))^n\over n!}e^{-\Lambda([u])x(3t-2s-r)}dx{\cal R}_{\rm typ}(du)drds.
\end{split}
\end{equation}
We start by noting that an I-segment arises as the intersection of two I-polygons, which are born at two different time instants, the maximum of which is the birth time of the I-segment and has a density $p_{\beta}(s)$ given by (\ref{margd}). However, in contrast to the planar case, I-segments in a spatial STIT tessellation can already have vertices in their relative interior at their time of birth. These vertices are generated by extant I-segments 'on the backside' of the carrying polygon, see Figure \ref{figpolygon} (left).  Given its length $\ell=x\in(0,\infty)$, its direction $\varphi=u\in{\cal S}_+^2$, its time of birth $\beta=s\in(0,t)$ and the birth time $\beta_{\rm carr}=r\in(0,s)$ of the carrying I-polygon, the number $N_{\rm birth}$ of vertices in the relative interior of the typical I-segment at the time of birth $s$ has, by (\ref{BOXPLUS}) and \textbf{[Linear sections]}, a Poisson distribution with parameter $\Lambda([u])x(s-r)$. Moreover, in the time interval $(s,t]$, on both sides of the I-polygon with birth time $s$ and `behind' the carrying I-polygon a Poisson distributed number of vertices appear in the relative interior of the segment, whose parameter is, by (\ref{BOXPLUS}) and again by \textbf{[Linear sections]}, $\Lambda([u])x(t-s)$, see Figure \ref{figpolygon} (right). Thus, the sum of these independent numbers is again Poisson distributed, but with parameter $3\Lambda([u])x(t-s)$. Adding the independent number $N_{\rm birth}$ leads to a Poisson distribution with parameter $\Lambda([u])x(3t-2s-r)$.

Recall now that we have calculated in Corollary \ref{cor2} the conditional joint density of length and birth time vector and, thus, mixing the Poisson distributed number of points with respect to this density and the directional distribution ${\cal R}_{\rm typ}$, we arrive at (\ref{helpeq}).

We show now that (\ref{helpeq}) is equivalent to (\ref{eqpn}). First, using $$\int\limits_0^\infty (\Lambda([u])x)^ne^{-\Lambda([u])sx}e^{-\Lambda([u])x(3t-2s-r)}dx={n!\over\Lambda([u])(3t-s-r)^{n+1}},$$ we see that (\ref{helpeq}) can be transformed into 
\begin{eqnarray}
\nonumber {\mathsf p}_n &=& \int\limits_0^t\int\limits_0^s\int\limits_{{\cal S}_+^2}{3s^2\over t^3}{(3t-2s-r)^n\over(3t-s-r)^{n+1}}{\cal R}_{\rm typ}(du)drds\\
\nonumber &=& \int\limits_0^t\int\limits_0^s{3s^2\over t^3}{(3t-2s-r)^n\over(3t-s-r)^{n+1}}drds .
\end{eqnarray}
 Applying now in the inner integral the substitution $b=1-r/s$, we obtain $${\mathsf p}_n=\int\limits_0^t\int\limits_0^1{3s^3\over t^3}{(3t-s(3-b))^n\over(3t-s(2-b))^{n+1}}dbds$$ and applying the similar substitution $a=1-s/t$, finally leads to (\ref{eqpn}).\hfill $\Box$
\end{proof}

\begin{proof}[Proof of Corollary \ref{cor4}]
The numbers shown in Corollary \ref{cor4} can immediately be calculated from the explicit formula in Theorem \ref{thm3}. To see it, write
\begin{eqnarray*}
\sum_{n=0}^\infty n\;{\mathsf p}_n &=& 3\int\limits_0^1\int\limits_0^1(1-a)^3\sum_{n=0}^\infty n{(3-(1-a)(3-b))^n\over(3-(1-a)(2-b))^{n+1}}db\, da\\
&=& 3\int\limits_0^1\int\limits_0^1 (1-a)(a(3-b)+b)db\, da=2
\end{eqnarray*}
and similarly
\begin{eqnarray*}
\sum_{n=0}^\infty n^2\;{\mathsf p}_n &=& 3\int\limits_0^1\int\limits_0^1(1-a)^3\sum_{n=0}^\infty n^2{(3-(1-a)(3-b))^n\over(3-(1-a)(2-b))^{n+1}}db\, da\\
&=& 3\int\limits_0^1\int\limits_0^1 a^2(2b^2-11b+15)-a(4b^2-10b-3)+b(1+2b)da\, db\\
&=& {71\over 3},
\end{eqnarray*}
which yields the variance value $71/3-2^2=59/3$.\hfill $\Box$
\end{proof}

\begin{proof}[Proof of Theorem \ref{thm2}]
Let the birth time $\beta=s<t$ of the typical I-segment and the birth time $\beta_{\rm carr}=r<s$ of the carrying I-polygon be given. At time $s$, the typical I-segment can only have X-type vertices in its relative interior (this is the number $N_{\rm birth}$ in the proof of Theorem \ref{thm3} above) and further vertices of this type can be created from `behind' the segment until time $t$. At time $s$ the I-segment is an edge (1-facet) of two cells. Vertices of type T in the interior of the I-segment can  appear in the time interval $(s,t)$ by further division of these two adjacent cells. After these observations the proof readily follows the lines of the proof of Theorem \ref{thm3} and for this reason the details are omitted.\hfill $\Box$
\end{proof}

\begin{proof}[Proof of Corollary \ref{cor5}]
 Corollary \ref{cor5} follows from Theorem \ref{thm2} by direct calculation. For example, the mean number of T-type vertices in the relative interior of the typical I-segment can be calculated as follows:
\begin{eqnarray*}
 \sum_{m=0}^\infty m\sum_{n=0}^\infty{\mathsf p}_{m,n} &=& 3\sum_{m=0}^\infty m \int\limits_0^1\int\limits_0^1 {(1-a)^3\over 1+a}\left({2a\over 1+a}\right)^m db\, da\\
 &=& 6\int\limits_0^1 a(1-a)da\int\limits_0^1db = 1. 
\end{eqnarray*}
For the second moment we find
\begin{eqnarray*}
 \sum_{m=0}^\infty m^2\sum_{n=0}^\infty{\mathsf p}_{m,n} &=& 3\sum_{m=0}^\infty m^3 \int\limits_0^1\int\limits_0^1 {(1-a)^3\over 1+a}\left({2a\over 1+a}\right)^m db\, da\\
 &=& 6\int\limits_0^1 a(1+3a)da\int\limits_0^1db = 9
\end{eqnarray*}
so that the variance equals $9-1^2=8$. Similar calculations yield mean $1$ and variance $11/3$ for the number of X-vertices in the relative interior of the typical I-segment. The value for the covariance is obtained from Corollary \ref{cor4} by calculating $59/3-11/3-8=8$, which completes the argument.\hfill $\Box$
\end{proof}

\section*{Acknowledgement}
The authors would like to thank Claudia Redenbach (Kaiserslautern) for providing the simulations in Figure \ref{figstit}. We also thank an anonymous referee for his or her valuable comments. VW and WN acknowledge the support from the DFG, grants WE 1899/3-1 and NA 247/6-1.

\end{document}